\documentclass[12pt,reqno,ctex]{amsart}
\usepackage{hyperref}
\usepackage{amssymb}
\usepackage{amsmath}
\usepackage{amsthm}
\usepackage{extarrows}
\usepackage{xcolor}
\usepackage{mathtools}
\usepackage{tikz}
\usepackage{caption}
\usepackage{subcaption}
\usepackage[numbers,sort&compress]{natbib}

\newtheorem{thm}{Theorem}[section]
\newtheorem{lemma}[thm]{Lemma}

\newtheorem{de}[thm]{Definition}

\numberwithin{equation}{section}
\usepackage{comment}
\usepackage{enumitem}

\begin{document}
	\title{Non-spectrality of Moran measures with consecutive digits}
	
	\author{Ya-Li Zheng, Wen-Hui Ai$^{\mathbf{*}}$}

	\address{Key Laboratory of High Performance Computing and Stochastic Information Processing (Ministry of Education of China), School of Mathematics and Statistics, Hunan Normal University, Changsha, Hunan 410081, China}
	
	\email{zheng92542021@163.com}
	\email{awhxyz123@163.com}
	
	\date{\today}
	
	
	\thanks{The research is supported in part by the NNSF of China (No.11831007).}
	\thanks{$^{\mathbf{*}}$Corresponding author}
	\maketitle \vskip 0.05in
	
	\begin{abstract}
		 Let $\rho=(\frac{p}{q})^{\frac{1}{r}}<1$ for some $p,q,r\in\mathbb{N}$ with $(p,q)=1$ and $\mathcal{D}_{n}=\{0,1,\cdot\cdot\cdot,N_{n}-1\}$, where $N_{n}$ is prime for all $n\in\mathbb{N}$, and denote $M=\sup\{N_{n}:n=1,2,3,\ldots\}<\infty$. The associated Borel probability measure
$$\mu_{\rho,\{\mathcal{D}_{n}\}}=\delta_{\rho\mathcal{D}_{1}}*\delta_{\rho^{2}\mathcal{D}_{2}}*\delta_{\rho^{3}\mathcal{D}_{3}}*\cdots$$
is called a Moran measure. Recently, Deng and Li proved that $\mu_{\rho,\{\mathcal{D}_{n}\}}$ is a spectral measure if and only if $\frac{1}{N_{n}\rho}$ is an integer for all $n\geq 2$.
In this paper, we prove that if $L^{2}(\mu_{\rho, \{\mathcal{D}_{n}\}})$ contains an infinite orthogonal exponential set, then there exist infinite positive integers $n_{l}$ such that $(q,N_{n_{l}})>1$. Contrastly, if $(q,N_{n})=1$ and $(p,N_{n})=1$ for all $n\in\mathbb{N}$, then there are at most $M$ mutually orthogonal exponential functions in $L^{2}(\mu_{\rho, \{\mathcal{D}_{n}\}})$ and $M$ is the best possible. If $(q,N_{n})=1$ and $(p,N_{n})>1$ for all $n\in\mathbb{N}$, then there are any number of orthogonal exponential functions in $L^{2}(\mu_{\rho, \{\mathcal{D}_{n}\}})$.
		
	\end{abstract}
	
	
	\section{Introduction}
	\
	
	A set $\Lambda$ in $\mathbb{R}^{n}$ is called a spectrum for a Borel probability measure $\mu$ on $\mathbb{R}^{n}$ with compact support if the corresponding exponential functions $E_\Lambda=\{e^{2\pi i\left\langle\lambda,x\right \rangle}:\lambda\in\Lambda\}$ forms an orthonormal basis for the Hilbert space $L^{2}(\mu)$. In this case, $\mu$ is called a spectral measure. Particularly, if $\mu$ is the normalized Lebesgue measure restricted on a subset
	$\Omega\subset\mathbb{R}$, then $\Omega$ is called a spectral set.

	In \cite{F}, Fuglede proposed the following famous conjecture: $\Omega$ is a spectral set if and only if $\Omega$ is a translational tile. Although the conjecture was disproved eventually (see \cite{KM} and \cite{T}) for the case that $\Omega\subset\mathbb{R}^{n}$ with $n\geq3$, it is still an open problem in $\mathbb{R}$ and $\mathbb{R}^{2}$. The conjecture has led to the development of the spectrality of probability measures. The first surprising discovery comes from \cite{JP}, where Jorgensen and Pedersen proved that the standard middle-fourth Cantor measure is a singular, non-atomic, spectral measure. They proved that the Cantor measure $\mu_{\rho}$ is a spectral measure if $\rho=\frac{1}{2k}$. For more general Bernoulli convolution  $\mu_{\rho}$, Hu and Lau \cite{HL} proved that the necessary and sufficient condition that $L^{2}(\mu_{\rho})$ contains an infinite exponential orthonormal set is $\rho=(\frac{p}{q})^{\frac{1}{n}}$ , where $n\geq1$, $p$ is odd and $q$ is even. Dai \cite{D1} completely settled the problem that the only spectral Bernoulli convolutions are of the contraction ratio $\frac{1}{2k}$. Furthermore, Dai, He and Lai \cite{DHL} showed that the $N$-Bernoulli convolution $\mu_{N}$ is a spectral measure if and only if $N|\rho^{-1}$. From then on, many other self-similar/self-affine/Moran spectral measures have been discovered, see \cite{AH14,AHL15,DHL19,AFL19,LDL22} and so on. However, a non-spectral measure $\mu$ belongs to one of the following three classes:
\begin{itemize}
\item  There are at most a finite number of orthogonal exponentials in $L^{2}(\mu)$.
\item  All the cardinality of orthogonal exponentials in $L^{2}(\mu)$ are bounded, but not uniform bounded.
\item  There exists an infinite set of orthogonal exponentials but none of such sets forms a basis for $L^{2}(\mu)$.
\end{itemize}
Deng et al. \cite{D2, WWDZ, WDJ} also gave some sufficient and necessary conditions under which that $L^{2}(\mu_{N})$ contains an infinite orthogonal exponential set.

	Moran-type measures are firstly raised by Strichartz \cite{SR} in 2000, who studied the following measure with the product structure of infinite convolution:
	\begin{equation*}
	\mu_{\{M_{k},\mathcal{D}_{k}\}}=\delta_{M_{1}^{-1}\mathcal{D}_{1}}*\delta_{M_{1}^{-1}M_{2}^{-1}\mathcal{D}_{2}}*\cdots,
	\end{equation*}
	where $\{M_{k}\}\subset M_{n}(\mathbb{R})$ is a sequence of expanding real matrices (all the eigenvalues of $M_{k}$ have moduli $>1$) and $\{\mathcal{D}_{k}\}\subset\mathbb{Z}^{n}$ is a sequence of finite
	digit sets with cardinality $\#\mathcal{D}_{k}$. In this case, $\delta_{E}:=\frac{1}{\#E}\sum_{e\in E}\delta_{e}$, $\#E$ is the cardinality of a set $E$, $\delta_{e}$ is the Dirac measure at $e$ and the convergence is in the weak sense.

In 2014, An and He \cite{AH14} studied the spectrality of Moran measure
	$\mu_{\{b_{k}, \mathcal{D}_{k}\}}=\delta_{{b_{1}}^{-1}\mathcal{D}_{1}}* \delta_{({b_{1}}{b_{2}})^{-1} D_{2}}*\cdots$ with consecutive digits $\mathcal{D}_{k}=\{0, 1, \dots, q_{k}-1\}$. They proved that $\mu_{\{b_{k},\mathcal{D}_{k}\}}$ is a spectral Moran measure if $q_k|b_k$. Recently, Deng and Li \cite{DL} obtained that $\mu_{\{b_{k},\mathcal{D}_{k}\}}$ is a spectral measure if and only if $q_k|b_k$ for all $k\geq 2$. However, there are a few results about the non-spectrality of Moran measures except the paper \cite{LW}. Li and Wu \cite{LW} considered the non-spectrality of Moran measures with two element digits. In this paper, we are focused on the non-spectrality of Moran measures with consecutive digits.
	
	
	 Let $\rho=(\frac{p}{q})^{\frac{1}{r}}<1$ for some $p,q,r\in\mathbb{N}$ with $(p,q)=1$ and $\mathcal{D}_{n}=\{0,1,2,\ldots,N_{n}-1\}$,
	 where $N_{n}$ is prime number for all $n\in\mathbb{N}$, and
	 \begin{equation}\label{1.2}
	M=\sup\{N_{n}:n=1,2,3,\ldots\}<\infty.
	\end{equation}
By Hutchinson \cite{Hut} and Strichartz \cite{SR}, there is a unique Moran measure
\begin{equation}\label{1.1}
	\mu_{\rho,\{\mathcal{D}_{n}\}}=\delta_{\rho\mathcal{D}_{1}}*\delta_{\rho^{2}\mathcal{D}_{2}}*\delta_{\rho^{3}\mathcal{D}_{3}}*\cdots.
	\end{equation}
We will always have the above assumptions in the following statement when there is no confusion. Throughout the paper, we make the convention that all fractions have the simplest form, that is for a fraction $\frac{p}{q}$  we have $\gcd(p,q)=1$. And $r$ is the smallest integer such that $(\frac{p}{q})^{\frac{1}{r}}\in\mathbb{Q}$ (for example, $\rho=(\frac{4}{9})^{\frac{1}{4}}=({\frac{2}{3}})^{\frac{1}{2}}$, we denote $r=2$).

In this paper, we obtain the following non-spectrality of Moran measures $\mu_{\rho,\{\mathcal{D}_{n}\}}$.		
	
\begin{thm}\label{thm1.1}
		Let $\mu_{\rho, \{\mathcal{D}_{n}\}}$ be defined by \eqref{1.1}.
		Then the following hold.

		$(\rm i)$ If $L^{2}(\mu_{\rho, \{\mathcal{D}_{n}\}})$ contains an infinite orthogonal exponential set, then there exist infinite positive integers $n_{l}$ such that $(q,N_{n_{l}})>1$.

		$(\rm ii)$  If $(q,N_{n})=1$ and $(p,N_{n})=1$ for all $n\in\mathbb{N}$, then there are at most $M$(defined by \eqref{1.2}) mutually orthogonal exponential functions in $L^{2}(\mu_{\rho,\{\mathcal{D}_{n}\}})$, and $M$ is the best possible.

		$(\rm iii)$ If $(q,N_{n})=1$ and $(p,N_{n})>1$ for all $n\in\mathbb{N}$, then there are any number of orthogonal exponential functions in $L^{2}(\mu_{\rho, \{\mathcal{D}_{n}\}})$, but cannot be infinite.
	\end{thm}

	We organize the paper as follows. In Section 2, we state some preliminary knowledge. In Section 3, we give a necessary condition for the existence of an infinite orthogonal exponential set in $L^{2}(\mu_{\rho, \{\mathcal{D}_{n}\}})$ and prove our main results.

	\section{Preliminaries}
	This section is devoted to giving some preliminary results and some basic concepts which will be used. We define the Fourier transformation of a probability measure $\mu$ in $\mathbb{R}^{n}$ by
	\begin{equation*}
	\hat{\mu}(\xi)=\int e^{-2\pi i\left\langle\xi,x\right\rangle}d\mu(x),
	 \xi\in\mathbb{R}^{n}.
	\end{equation*}
	Let $\mu$:=$\mu_{\rho, \{\mathcal{D}_{n}\}}$ be given by \eqref{1.1}, then
	\begin{equation}\label{2.0}
	\hat{\mu}(\xi)=\hat{\mu}_{\rho, \{\mathcal{D}_{n}\}}(\xi)=\prod_{n=1}^{\infty}M_{\mathcal{D}_{n}}(\rho^{n}\xi), \xi\in\mathbb{R},
	\end{equation}
	where $M_{\mathcal{D}_{n}}(\xi)=\frac{1}{N_{n}} \sum_{j=0}^{N_{n}-1} e^{-2 \pi i j \xi}$ is the mask polynomial of $\mathcal{D}_{n}$.

	Let $\mathcal{Z}(f)=\{\xi:f(\xi)=0\}$ be the set of zeros of $f$. We obtain that
	\begin{equation}\label{2.1}
	\mathcal{Z}\left(\hat{\mu}\right)
=\bigcup_{n=1}^{\infty}\mathcal{Z}\left( M_{\mathcal{D}_{n}}(\rho^{n}\xi)\right)
= \bigcup_{n=1}^{\infty} \frac{\rho^{-n}a_{n}}{N_{n}},\ a_{n}\in\mathbb{Z}\backslash N_{n}\mathbb{Z}.
	\end{equation}

	For $\lambda_{i}\neq\lambda_{j}\in\mathbb{R}$, the orthogonal condition is
	\begin{equation*}
	0=\left\langle e^{2\pi i\left\langle\lambda_{i},x\right\rangle},e^{2\pi i\left\langle\lambda_{j},x\right\rangle}\right\rangle_{L^{2}(\mu)}=\int e^{2\pi i\left\langle\lambda_{i}-\lambda_{j},x\right\rangle}d\mu(x)=\hat{\mu}(\lambda_{i}-\lambda_{j}).
	\end{equation*}
For a countable set $\Lambda\subset\mathbb{R}$, it is easy to see that $E_{\Lambda}=\{e^{2\pi i\left\langle\lambda,x\right\rangle}:\lambda\in\Lambda\}$ is an orthonormal family of $L^{2}(\mu)$ if and only if $$(\Lambda-\Lambda)\setminus\{0\}\subset\mathcal{Z}(\hat{\mu}).$$
We say $\Lambda$ is a bi-zero set of $\mu$ if $(\Lambda-\Lambda)\setminus\{0\}\subset\mathcal{Z}(\hat{\mu})$.
$\Lambda$ is called a spectrum of $\mu$ if $E_{\Lambda}$ is an orthogonal basis for $L^{2}(\mu)$.
	

	To prove our results, we need the following lemmas which will play a crucial role in the proofs.
	
	\begin{lemma} (\cite[Lemma 2.5]{D2})\label{lem3.1}
		Assume that $a\in\mathbb{R}$ admits a minimal integer polynomial $qx^{r}-p(r>1)$ and satisfies $b_{1}a^{k}+b_{2}a^{j}=b_{3}a^{u}$,where $k,j,u\geq0$ are  nonnegative integers and $b_{1},b_{2},b_{3}\in\mathbb{Z}\setminus\{0\}$. Then $k\equiv j\equiv u\pmod r$.
	\end{lemma}

	We are now in a position to prove an intresting result. The specific is as follows.
	\begin{lemma}\label{lem3.2}
		Let $\mu$ be defined by \eqref{1.1}. If $(\Lambda-\Lambda)\setminus{0}\subset\mathcal{Z}(\hat{\mu})$, then the corresponding $N_{n}$ must be equal.
	\end{lemma}
	
	\begin{proof}
		Let $\lambda_{1},\ \lambda_{2}\in\mathcal{Z}(\hat{\mu})\cap\Lambda$. From \eqref{2.1} there exists some $\lambda_{s}=\frac{\rho^{-n_{s}}a_{n_{s}}}{N_{n_{s}}}\in\mathcal{Z}(\hat{\mu})$ such that $\lambda_{2}-\lambda_{1}=\lambda_{s}$, i.e.,
		\begin{equation}\label{3.2}
		\rho^{-n_{2}}\frac{a_{n_{2}}}{N_{n_{2}}}-\rho^{-n_{1}}\frac{a_{n_{1}}}{N_{n_{1}}}=\rho^{-n_{s}}\frac{a_{n_{s}}}{N_{n_{s}}}.
		\end{equation}
		By Lemma \ref{lem3.1} we have $n_{1}\equiv n_{2}\equiv n_{s}\pmod r$. Write $n_{1}=n_{1}^{\prime}r+k,\ n_{2}=n_{2}^{\prime}r+k,\ n_{s}=n_{s}^{\prime}r+k$,\  $n_{1}^{\prime},\ n_{2}^{\prime},\ n_{s}^{\prime}\in\mathbb{Z},\ 0\leq k\leq r-1$. \eqref{3.2} can be expressed by
		\begin{equation*}
		(\frac{q}{p})^{n_{2}^{\prime}}\frac{a_{n_{2}}}{N_{n_{2}}}-(\frac{q}{p})^{n_{1}^{\prime}}\frac{a_{n_{1}}}{N_{n_{1}}}=(\frac{q}{p})^{n_{s}^{\prime}}\frac{a_{n_{s}}}{N_{n_{s}}}.
		\end{equation*}
		Without loss of generality, we assume that $n_{1}^{\prime}<n_{2}^{\prime}$. Then
		\begin{equation}\label{3.3}
		(\frac{q}{p})^{n_{2}^{\prime}-n_{1}^{\prime}}\frac{a_{n_{2}}}{N_{n_{2}}}-\frac{a_{n_{1}}}{N_{n_{1}}}=(\frac{q}{p})^{n_{s}^{\prime}-n_{1}^{\prime}}\frac{a_{n_{s}}}{N_{n_{s}}}.
		\end{equation}

		Suppose $n_{s}^{\prime}>n_{1}^{\prime}$, according to the relation among $p$, $q$ and $N_{n}$, we divide this case into three cases. Then we can prove that $N_{n_{2}},N_{n_{1}},N_{n_{s}}$ must be equal by using reduction to absurdity.

		\textbf{Case 1}: $(p,N_{n})=1$ and $(q,N_{n})=1$ for all $n\in\mathbb{N}$.
		
		If $N_{n_{1}},\ N_{n_{2}},\ N_{n_{s}}$ are respectively different, \eqref{3.3} implies that
		\begin{equation}\label{3.4}
		p^{n_{s}^{\prime}-n_{1}^{\prime}}q^{n_{2}^{\prime}-n_{1}^{\prime}}a_{n_{2}}N_{n_{1}}-p^{n_{2}^{\prime}+n_{s}^{\prime}-2n_{1}^{\prime}}a_{n_{1}}{N_{n_{2}}}=p^{n_{2}^{\prime}-n_{1}^{\prime}}q^{n_{s}^{\prime}-n_{1}^{\prime}}a_{n_{s}}\frac{N_{n_{2}}N_{n_{1}}}{N_{n_{s}}}.
		\end{equation}
		Obviously, by the assumption, all $p,\ q,\ N_{n}$ are co-prime respectively and $a_{n_{s}}\in\mathbb{Z}\setminus N_{n_{s}}\mathbb{Z}$, hence the left hand of \eqref{3.4} is an integer and the right hand is a fraction. This is a contradiction, so at least two of  $N_{n_{1}},\ N_{n_{2}},\ N_{n_{s}}$ are equal. We can assume $N_{n_{1}}=N_{n_{2}}\neq N_{n_{s}}$, then \eqref{3.3} implies that
		\begin{equation}\label{3.5}
		p^{n_{s}^{\prime}-n_{1}^{\prime}}q^{n_{2}^{\prime}-n_{1}^{\prime}}a_{n_{2}}-p^{n_{2}^{\prime}+n_{s}^{\prime}-2n_{1}^{\prime}}a_{n_{1}}=p^{n_{2}^{\prime}-n_{1}^{\prime}}q^{n_{s}^{\prime}-n_{1}^{\prime}}a_{n_{s}}\frac{N_{n_{2}}}{N_{n_{s}}}.
		\end{equation}
		The same argument as before shows that \eqref{3.5} is a contradiction. Moreover, we can also get contradictions for the cases $N_{n_{1}}=N_{n_{s}}\neq N_{n_{2}}$ and $N_{n_{s}}=N_{n_{2}}\neq N_{n_{1}}$. Hence $N_{n_{1}}=N_{n_{2}}=N_{n_{s}}$.

		\textbf{Case 2}: $(p,N_{n})>1$ and $(q,N_{n})=1$ for all $n\in\mathbb{N}$.

		If $N_{n_{1}},\ N_{n_{2}},\ N_{n_{s}}$ are respectively different, \eqref{3.3} implies that
		\begin{equation}\label{3.6}
		p^{n_{2}^{\prime}-n_{1}^{\prime}}a_{n_{1}}\frac{N_{n_{s}}}{N_{n_{1}}}+p^{n_{2}^{\prime}-n_{s}^{\prime}}q^{n_{s}^{\prime}-n_{1}^{\prime}}a_{n_{s}}=q^{n_{2}^{\prime}-n_{1}^{\prime}}a_{n_{2}}\frac{N_{n_{s}}}{N_{n_{2}}}.
		\end{equation}
		If $n_{2}^{\prime}\geq n_{s}^{\prime}$, the right hand of \eqref{3.6} is a fraction under the assumption $(q,N_{n})=1$ and $(p,N_{n})>1$, contracting that the left is an integer. If $n_{2}^{\prime}<n_{s}^{\prime}$, \eqref{3.6} can be expressed into
		\begin{equation*}
		a_{n_{2}}q^{n_{2}^{\prime}-n_{1}^{\prime}}\frac{p^{n_{s}^{\prime}-n_{2}^{\prime}}}{N_{n_{2}}}-a_{n_{1}}\frac{p^{n_{s}^{\prime}-n_{1}^{\prime}}}{N_{n_{1}}}=q^{n_{s}^{\prime}-n_{1}^{\prime}}\frac{a_{n_{s}}}{N_{n_{s}}}.
		\end{equation*}
		The above is also a contradiction, so at least two of  $N_{n_{1}},\ N_{n_{2}},\ N_{n_{s}}$ are equal. Similarly, suppose $N_{n_{1}}=N_{n_{s}}\neq N_{n_{2}}$, then \eqref{3.3} implies that
		\begin{equation*}
		p^{n_{2}^{\prime}-n_{1}^{\prime}}a_{n_{1}}+p^{n_{2}^{\prime}-n_{s}^{\prime}}q^{n_{s}^{\prime}-n_{1}^{\prime}}a_{n_{s}}=q^{n_{2}^{\prime}-n_{1}^{\prime}}a_{n_{2}}\frac{N_{n_{s}}}{N_{n_{2}}}.
		\end{equation*}
		For the same reason, it is a contradiction. Hence $N_{n_{1}}=N_{n_{2}}=N_{n_{s}}$.

		\textbf{Case 3}: $(p,N_{n})=1$ and $(q,N_{n})>1$ for all $n\in\mathbb{N}$.

		If $N_{n_{1}},\ N_{n_{2}},\ N_{n_{s}}$ are respectively different, \eqref{3.3} implies that
		\begin{equation}\label{3.7}
		q^{n_{2}^{\prime}-n_{1}^{\prime}}a_{n_{2}}-p^{n_{2}^{\prime}-n_{s}^{\prime}}a_{n_{s}}N_{n_{2}}\frac{q^{n_{s}^{\prime}-n_{1}^{\prime}}}{N_{n_{s}}}=p^{n_{2}^{\prime}-n_{1}^{\prime}}a_{n_{1}}\frac{N_{n_{2}}}{N_{n_{1}}}.
		\end{equation}
		In the same way, we can get that the above equation is a contradiction whether $n_{2}^{\prime}\geq n_{s}^{\prime}$. Therefore, at least two of  $N_{n_{1}},\ N_{n_{2}},\ N_{n_{s}}$ are equal, we can assume $N_{n_{2}}=N_{n_{s}}\neq N_{n_{1}}$. Then \eqref{3.7} implies that
		\begin{equation}\label{3.8}
		q^{n_{2}^{\prime}-n_{1}^{\prime}}a_{n_{2}}-p^{n_{2}^{\prime}-n_{s}^{\prime}}a_{n_{s}}q^{n_{s}^{\prime}-n_{1}^{\prime}}=p^{n_{2}^{\prime}-n_{1}^{\prime}}a_{n_{1}}\frac{N_{n_{2}}}{N_{n_{1}}},
		\end{equation}
		obviously, a contradiction. Hence $N_{n_{1}}=N_{n_{2}}=N_{n_{s}}$.

		When $n_{s}^{\prime}\leq n_{1}^{\prime}$, the proof is similar to that of case $n_{s}^{\prime}>n_{1}^{\prime}$, so we omit it here. Therefore, $N_{n_{1}}=N_{n_{2}}=N_{n_{s}}$ as requried.
	\end{proof}

	\section{Proof of Theorem \ref{thm1.1}}
	In the sequel, the whole proof of Theorem \ref{thm1.1} will be divided into the following three parts.

	 We firstly give a necessary condition for the existence of an infinite orthogonal exponential set in $L^{2}(\mu)$, i.e. Theorem \ref{thm1.1} $(\rm i)$.

	\begin{proof}[Proof of Theorem \ref{thm1.1} $(\rm i)$]
		
		Let $E_{\Lambda}$ be an infinite orthonomal set with $0\in\Lambda$, then $\Lambda\setminus\{0\}\subseteq\mathcal{Z}(\hat{\mu})$. For any $\lambda\in\mathcal{Z}(\hat{\mu})$, there exists $k\in\mathbb{N}$, $B_{k}\in\mathbb{Z}\setminus N_{k}\mathbb{Z}$ such that $\lambda=\frac{B_{k}}{\rho^{k}N_{k}}$ by \eqref{2.1}. Furthermore, there exist integers $l_{k}\geq 0$ and $b_{k}\in\mathbb{Z}$ such that $B_{k}=q^{l_{k}}b_{k}$ with $q\nmid b_{k}$. In this way, we have $\frac{B_{k}}{\rho^{k}N_{k}}=\frac{q^{l_{k}}b_{k}}{p^{k}N_{k}}=\frac{p^{l_{k}}b_{k}}{N_{k}\rho^{k+rl_{k}}}$ with $q\nmid b_{k}$ and $N_{k}\nmid q^{l_{k}}b_{k}$.

		Let $k_{1}$ be the smallest positive integer such that $\frac{l_{k_{1}}}{\rho^{k_{1}}N_{k_{1}}}\in(\Lambda-\Lambda)\setminus\{0\}\subset\mathcal{Z}(\hat{\mu})$, where some $l_{k_{1}}\in\mathbb{Z}\setminus q\mathbb{Z}$. By the property of bi-zero set, there are $\lambda_{1},\ \lambda_{0}\in\Lambda$ such that $\lambda_{1}-\lambda_{0}=\frac{l_{k_{1}}}{\rho^{k_{1}}N_{k_{1}}}$. For any $\lambda\in\Lambda$, \eqref{2.1} implies that there must exist $\lambda^{\prime}\in\mathcal{Z}(\hat{\mu})$ such that $\lambda_{1}-\lambda=\lambda^{\prime}$, i.e., $\frac{l_{k_{1}}}{\rho^{k_1}N_{k_{1}}}-(\lambda-\lambda_{0})=\lambda^{\prime}$. Write $\lambda=\lambda_{0}+\frac{p^{l_{k}}b_{k}}{N_{k}\rho^{k+rl_{k}}}$ and $\lambda^{\prime}=\frac{p^{l_{t}}b_{t}}{N_{t}\rho^{t+rl_{t}}}$ with the properties: $b_{k}\in(\mathbb{Z}\setminus p\mathbb{Z})\setminus  N_{k}\mathbb{Z}$, $b_{t}\in(\mathbb{Z}\setminus p\mathbb{Z})\setminus N_{t}\mathbb{Z}$,\ $t,k\in\mathbb{N}$,\ $l_{k},l_{t}\geq 0$. Then we have
		\begin{equation}\label{3.9}
		\frac{l_{k_{1}}}{\rho^{k_{1}}N_{k_{1}}}-\frac{p^{l_{k}}b_{k}}{N_{k}\rho^{k+rl_{k}}}=\frac{p^{l_{t}}b_{t}}{N_{t}\rho^{t+rl_{t}}}.
		\end{equation}
		From Lemma \ref{lem3.2}, $N_{k_{1}}=N_{k}=N_{t}$.
		By Lemma \ref{lem3.1} we have $k-k_{1}\equiv t-k_{1}\equiv 0 \pmod r$. Let $k+rl_{k}=nr+k_{1}$, $t+rl_{t}=ur+k_{1}$, by the definition of $k_{1}$, it is easy to see $n,u\geq 0$. Then \ref{3.9} becomes
		\begin{equation}\label{3.11}
		l_{k_{1}}-p^{l_{k}-n}b_{k}q^{n}=p^{l_{t}-u}b_{t}q^u.
		\end{equation}	
		If $n=0$, then $\lambda=\lambda_{0}+\frac{p^{l_{k}}b_{k}}{N_{k}\rho^{k_{1}}}$. If $n>0$, we claim that $u=0$. Otherwise, if $u>0$, \eqref{3.11} means $l_{k_{1}}=p^{l_{k}-n}b_{k}q^{n}+p^{l_{t}-u}b_{t}q^u$, without loss of generality, we assume that $n>u$, we have $l_{k_{1}}=q^{u}{p^{l_{k}-n}b_{k}q^{n-u}+p^{l_{t}-u}b_{t}}$, then the above equality implies $q|l_{k_{1}}$, a contradiction. Hence $u=0$.
We can rewrite \eqref{3.11} as
		\begin{equation}\label{3.12}
		p^{n}(l_{k_{1}}+p^{l_{t}}b_{t})=p^{l_{k}}{b_{k}}q^{n}.
		\end{equation}
It is easy to see $p^{n}|p^{l_{k}}{b_{k}}$. Let $\ s_{k}\in\mathbb{Z}$ such that $p^{l_{k}}{b_{k}}=s_{k}p^{n}$. Then $$\lambda=\lambda_{0}+\frac{p^{n}s_{k}}{N_{k}\rho^{k_{1}+nr}}=\lambda_{0}+\frac{q^{n}s_{k}}{N_{k}\rho^{k_{1}}}.$$
 Therefore, there exsits non-zero integers $a_{j}$ such that
		\begin{equation*}
		\Lambda=\{\lambda_{0}\}\cup\{{\lambda_{j}\}_{j=1}^{+\infty}},\  \lambda_{j}=\lambda_{0}+\frac{a_{j}}{N_{k}\rho^{k_{1}}}.
		\end{equation*}
		 Choose $s,t>0$ such that $N_{k}|(a_{s}-a_{t})$ and denote $a_{s}-a_{t}=m_{k}N_{k}$,\  $m_{k}\in\mathbb{Z}$. Then \eqref{2.1} implies that there exists $B_{n_{s,t}}\in\mathbb{Z}\setminus N_{n_{s,t}}\mathbb{Z},\ n_{s,t}\in\mathbb{Z}$ such that
		\begin{equation*}
		\frac{m_{k}}{\rho^{k_{1}}}=\lambda_{s}-\lambda_{t}=\frac{B_{n_{s,t}}}{N_{n_{s,t}}\rho^{n_{s,t}}}.
		\end{equation*}
		By Lemma \ref{lem3.1}, $n_{s,t}-k_{1}=ru_{s,t}$ for some integer $u_{s,t}$. Hence
		\begin{equation}\label{3.13}
		(\frac{p}{q})^{u_{s,t}}=\frac{B_{n_{s,t}}}{m_{k}N_{n_{s,t}}}.
		\end{equation}

		The definition of $k_{1}$ implies $u_{s,t}\geq0$. Note that $N_{n_{s,t}}\nmid B_{n_{s,t}}$, we see that $u_{s,t}>0$, so \eqref{3.13} means that $(q,N_{n_{s,t}})>1$. In the same manner, we can find infinite $n$ such that $(q,N_{n})>1$. The proof is complete.
	\end{proof}

	It follows from Theorem \ref{thm1.1} $(\rm i)$ that if $(q,N_{n})=1$, then every orthogonal set of exponential functions in $L^{2}(\mu)$ is finite. One can naturally ask: What is the maximal cardinality of the orthogonal exponential functions in $L^{2}(\mu)$? Next, we can answer this question.

	Let $\rho=\left(\frac{p}{q}\right)^{\frac{1}{r}}$  for some $p,q,r\in\mathbb{N}$. Then for any $\xi\in\mathbb{R}$,
	\begin{equation*}
	\hat{\mu}(\xi)=\prod_{i=1}^{r} \prod_{j=0}^{\infty}M_{\mathcal{D}_{jr+i}}\left(\left(\frac{p}{q}\right)^{j} \rho^{i} \xi\right).
	\end{equation*}
	Let $\hat{\nu}_{i}(\xi)=\prod_{j=0}^{\infty} M_{\mathcal{D}_{jr+i}}\left(\left(\frac{p}{q}\right)^{j} \rho^{i} \xi\right)$ for $1\leq i\leq r$. Therefore $\mu=\nu_{1}\ast \nu_{2}\ast\cdots*\nu_{r}$ and
	\begin{equation}\label{3.15}
	\mathcal{Z}\left(\hat{\mu}\right)=\bigcup_{i=1}^{r} \mathcal{Z}\left(\hat{\nu}_{i}\right),\ \mathcal{Z}\left(\hat{\nu}_{i}\right)=\rho^{-i} \bigcup_{j=0}^{\infty}\left(\frac{q}{p}\right)^{j} \frac{a_{j}}{N_{j}} ,\ a_{j}\in\mathbb{Z}\setminus N_{j}\mathbb{Z}.
	\end{equation}
	Since $N_{n}$ is prime and $(q,N_{n})=1$, \eqref{3.15} becomes
	\begin{equation}\label{3.16}
\mathcal{Z}\left(\hat{\mu}\right)\subset\bigcup_{i=1}^{r} \rho^{-i} \bigcup_{j=0}^{\infty}\frac{b_{j}}{p^{j}N_{j}},\ 1\leq i\leq r,\ b_{j}\in\mathbb{Z}\setminus N_{j}\mathbb{Z}.
	\end{equation}
	
	\begin{lemma}\label{lem3.3}
		Let $\rho=(\frac{p}{q})^{\frac{1}{r}}$ for some $p,q,r\in\mathbb{N}$ with $(p,q)=1$. Suppose that  $(q,N_{n})=1$ for all $n$, then $E_\Lambda$ is an orthogonal set of exponential functions in $L^{2}(\mu)$ if and only if there exists some $i\in\{1,2,...,r\}$ such that $(\Lambda-\Lambda)\backslash\{0\}\subset\mathcal{Z}(\hat{\nu_{i}})$, where $\mathcal{Z}\left(\hat{\nu}_{i}\right)$ is given by \eqref{3.15}.	
	\end{lemma}
	\begin{proof}
		It is easy to see that $E_{\Lambda}$ is an orthogonal set of exponential functions in $L^2(\hat{\mu})$ if  $(\Lambda-\Lambda)\backslash\{0\}\subset\mathcal{Z}(\hat{\nu_{i}})$. Assume that $(\Lambda-\Lambda)\backslash\{0\}\subset\mathcal{Z}(\hat{\nu_{i}})$ for some $i\in\{1,2,...,r\}$. Then the sufficiency follows from \eqref{2.1}.
		
		We now turn to prove the necessity. Assume that $E_{\Lambda}$ is an orthogonal set of exponential functions in $L^2(\hat{\mu})$. Then $(\Lambda-\Lambda)\backslash\{0\}\subset\mathcal{Z}(\hat{\mu})$.
		In fact, by \eqref{3.16}, we can write $\lambda_{1}=\frac{\rho^{-m_{1}}b_{j_{1}}}{p^{j_{1}}N_{j_{1}}}$ and $\lambda_{2}=\frac{\rho^{-m_{2}}b_{j_{2}}}{p^{j_{2}}N_{j_{2}}}$, where $b_{j_{1}}\in\mathbb{Z}\setminus N_{j_{1}}\mathbb{Z},\ b_{j_{2}}\in\mathbb{Z}\setminus N_{j_{2}}\mathbb{Z}$ and $1\leq m_{1},m_{2}\leq r$. Then there exists $\lambda^{\prime}=\frac{\rho^{-m^{\prime}}b_{j^{\prime}}}{p^{j^{\prime}}N_{j^{\prime}}}\in\mathcal{Z}(\hat{\mu})$ with $1\leq m^{\prime}\leq r,\ b_{j^{\prime}}\in\mathbb{Z}\setminus N_{j^{\prime}}\mathbb{Z}$ such that $\lambda_{1}-\lambda_{2}=\lambda^{\prime}$. Let $L=\max\{m_{1},\ m_{2},\ m^{\prime}\}$, then we have
		\begin{equation}\label{3.17}
		\rho^{-m_{1}}b_{j_{1}}p^{j_{2}+j_{3}}N_{j_{2}}N_{j_{3}}-\rho^{-m_{2}}b_{j_{2}}p^{j_{1}+j_{3}}N_{j_{1}}N_{j_{3}}=\rho^{-m_{3}}b_{j_{3}}p^{j_{1}+j_{2}}N_{j_{1}}N_{j_{2}}.
		\end{equation}
		If $r>1$, by Lemma \ref{lem3.1} we have $L-m_{1}\equiv L-m_{2}\equiv L-m^{\prime}\pmod r$ and thus $m_{1}\equiv m_{2}\equiv m^{\prime}\pmod r$. Together with $1\leq m_{1},\ m_{2},\ m^{\prime}\leq r$, we have $m_{1}=m_{2}=m^{\prime}$. If $r=1$, it is trivial. Therefore, $\lambda_{1},\ \lambda_{2},\ \lambda_{1}-\lambda_{2}$ belong to the same $\mathcal{Z}(\hat{\nu}_{i})$ for some $1\leq i\leq r$. Hence the proof is complete.
	\end{proof}
	
With the help of above lemma, we can now prove  Theorem \ref{thm1.1} $(\rm ii)$.
	\begin{proof}[Proof of Theorem \ref{thm1.1} $(\rm ii)$]
		We first prove that there are at most $M$ mutually orthogonal exponential functions in $L^{2}(\mu)$.
		Suppose on the contrary that ${\rm{\# }}\Lambda\geq M+1$. Let $\Lambda=\{0,\lambda_{1},\lambda_{2},...,\lambda_{M}\}$ be an orthogonal set for $\mu$, where $\lambda_{j}=\frac{\rho^{-i}b_{l_{j}}}{p^{l_{j}}N_{l_{j}}},\ 1\leq i\leq r,\ 1\leq j\leq M,\ b_{l_{j}}\in \mathbb{Z}\setminus N_{l_{j}}\mathbb{Z}$. According to Lemma \ref{lem3.2}, all $\{N_{l_{j}},\ 1\leq j\leq M\}$ are equal and write it as $N$. Let $m=\max\{l_{j},\ 1\leq j\leq M\}$, then
		\begin{equation*}
		\Lambda\setminus\{0\}=\frac{\rho^{-i}}{p^{m}N}\{p^{m-l_{j}}b_{l_{j}},\ 1\leq j\leq M, b_{l_{j}}\in \mathbb{Z}\setminus N\mathbb{Z}\}:=\frac{\rho^{-i}}{p^{m}N}\Lambda^{\prime}.
		\end{equation*}
		
		We first consider the case $N=M$. Since $(p,N_{n})=1,\ (q,N_{n})=1$ and $M$ is a prime number, we have
		$
		M\nmid p^{m-l_{j}}b_{l_{j}},
		$
		that is
		$
		p^{m-l_{j}}b_{l_{j}}\neq0\pmod M.
		$
		Together with $\#\Lambda^{\prime}=M$, we see that there exists at least two different $j_{1},\ j_{2}\in\{1,2,\ldots,M\}$, such that
		\begin{equation*}
		p^{m-l_{j_{1}}}b_{l_{j_{1}}}\equiv p^{m-l_{j_{2}}}b_{l_{j_{2}}}\pmod M.
		\end{equation*}
		By lemma \ref{lem3.2}, we have
		\begin{equation*}	
		\lambda_{j_{1}}-\lambda_{j_{2}}=\frac{\rho^{-i}}{Np^{m}}(b_{l_{j_{1}}}p^{m-l_{j_{1}}}-b_{l_{j_{2}}}p^{m-l_{j_{2}}})=\frac{\rho^{-i}kM}{p^{m}N}\notin\mathcal{Z}(\hat{\nu}_{i}),
		\end{equation*}
		where $k\in\mathbb{Z}\setminus\{0\}$. In fact, if there  exists some $\lambda=\frac{\rho^{-i}a_{j}}{p^{j}N}\in\mathcal{Z}(\hat{\nu}_{i})$  such that $\lambda_{j_{1}}-\lambda_{j_{2}}=\lambda$, then we have
		\begin{equation*}
		\frac{\rho^{-i}kM}{p^{m}N}=\frac{\rho^{-i}a_{j}}{p^{j}N},\ \text{i.e.}, \ p^{j}kM=a_{j}p^{m}.
		\end{equation*}
		 If $m\geq j$, we have $a_{j}p^{m-j}=kM$, which contradicts $M\nmid a_{j}$; if $m>j$, we have $p^{j-m}kM=a_{j}$, which also contradicts $M\nmid a_{j}$.

		We now turn to the case $N<M$. Similar to the above,  we have
		$
		N\nmid p^{m-l_{j}}b_{l_{j}},
		$
		that is
		$
		p^{m-l_{j}}b_{l_{j}}\neq0\pmod N.
		$
		Together with $\#\Lambda^{\prime}=M\geq N+1$, we see that there exist at least two different $j_{1},j_{2}\in\{1,2,...,M\}$, such that
		\begin{equation*}
		p^{m-l_{j_{1}}}b_{l_{j_{1}}}\equiv p^{m-l_{j_{2}}}b_{l_{j_{2}}}\pmod N.
		\end{equation*}
		By lemma \ref{lem3.3}, we have
		\begin{equation*}	
		\lambda_{j_{1}}-\lambda_{j_{2}}=\frac{\rho^{-i}}{Np^{m}}(b_{l_{j_{1}}}p^{m-l_{j_{1}}}-b_{l_{j_{2}}}p^{m-l_{j_{2}}})=\frac{\rho^{-i}kN}{p^{m}N}\notin\mathcal{Z}(\hat{\nu}_{i}),
		\end{equation*}
		where $k\in\mathbb{Z}\setminus\{0\}$. In fact, if there  exists $\lambda=\frac{\rho^{-i}a_{j}}{p^{j}N}\in\mathcal{Z}(\hat{\nu}_{i})$  such that $\lambda_{j_{1}}-\lambda_{j_{2}}=\lambda$, then we have
		\begin{equation*}
		\frac{\rho^{-i}kN}{p^{m}N}=\frac{\rho^{-i}a_{j}}{p^{j}N},\ \text{i.e.,} \ p^{j}kN=a_{j}p^{m}.
		\end{equation*}
		 If $m\geq j$, we have $a_{j}p^{m-j}=kN$, which contradicts $N\nmid a_{j}$; if $m>j$, we have $p^{j-m}kN=a_{j}$, contradicting $N\nmid a_{j}$.

		Hence, there are at most $M$ mutually orthogonal exponential functions in $L^{2}(\mu)$.

		In the following, we prove that $M$ is the best possible by constructing an orthogonal set of exponential functions with cardinality equal to $M$. let
		\begin{equation*}
		\Lambda_{0}=\{0\}\cup\{\frac{j}{\rho^{t}N_{t}}:1\leq j\leq M-1\}
		\end{equation*}
		for some $t$ such that $N_{t}=M=\sup\{N_{n}:n=1,2,3,\ldots\}<\infty$. It is easy to check that $(\Lambda_{0}-\Lambda_{0})\setminus\{0\}\subset\mathcal{Z}(\hat\nu_{i})$. Thus, $E_{\Lambda_{0}}$ is an orthogonal set of exponential functions in $L^{2}(\mu)$ by Lemma \ref{lem3.3}. Hence $M$ is the best possible.
	\end{proof}

We proceed to show there are any number of orthonomal exponential functions in $L^{2}(\mu)$ under the assumption that $(q,N_{n})=1$ and $N_{n}|p$.
We firstly introduce multiplicative order \cite{Z} that will be used in our proof.
	\begin{de}(\cite{Z})
		Let $n>0$, $(a,n)=1$ and let $s$ be the smallest positive integer such that $a^s\equiv 1\pmod n$, then $s$ is called the multiplicative order of $a\pmod n$.
	\end{de}
	
	\begin{proof}[Proof of Theorem \ref{thm1.1} $(\rm iii)$]
		Under the assumptions of Theorem \ref{thm1.1}, fix $1\leq i\leq r$, for any given positive integer $\alpha$, let
		\begin{equation*}
		\Lambda^{*}=\{\lambda_{n}=\frac{q^{(\alpha+n)s}}{N_{n}}(\frac{q}{p})^{ns+\frac{i}{r}}:1\leq n\leq \alpha\}\subset\mathcal{Z}(\hat{\nu}_{i}),
		\end{equation*}
		where $s$ is the multiplicative order of $q\pmod {N_{n}}$, and $\mathcal{Z}(\hat{\nu}_{i})$ is defined as \eqref{3.15}.

		Without loss of generality, we assume that $i=1$. For any integers $n_{1}, n_{2}$ with $1\leq n_{2}<n_{1}\leq \alpha$, we have
		\begin{align*}
		\lambda_{n_{1}}-\lambda_{n_{2}}&=\frac{q^{(\alpha+n_{1})s}}{N_{n_{1}}}(\frac{q}{p})^{n_{1}s+\frac{1}{r}}-\frac{q^{(\alpha+n_{2})s}}{N_{n_{2}}}(\frac{q}{p})^{n_{2}s+\frac{1}{r}}\\
		&=\frac{1}{N_{n_{1}}}(\frac{q}{p})^{n_{1}s+\frac{1}{r}}(q^{(\alpha+n_{1})s}-q^{(\alpha+n_{2})s}(\frac{p}{q})^{(n_{1}-n_{2})s})    (\rm{by\;Lemma}\;\ref{lem3.2})\\
		&=\frac{1}{N_{n_{1}}}(\frac{q}{p})^{n_{1}s+\frac{1}{r}}(q^{(\alpha+n_{1})s}-q^{(\alpha+2n_{2}-n_{1})s}p^{(n_{1}-n_{2})s}).\\
		\end{align*}
		In order to prove that $\lambda_{n_{1}}-\lambda_{n_{2}}\in\mathcal{Z}(\hat{\nu}_{1})$, we only need to prove that
		\begin{equation}\label{3.18}
		(q^{(\alpha+n_{1})s}-q^{(\alpha+2n_{2}-n_{1})s})p^{(n_{1}-n_{2})s}\ne 0\pmod {N_{n_{1}}}.
		\end{equation}	
		Due to $q^{s}\pmod {N_{n_{1}}}=1$, then
		\begin{equation}\label{3.19}
		q^{ks}\pmod{N_{n_{1}}}=1
		\end{equation}
		holds for any positive integer $k$.

		We cliam that for any two positive integers $a,b$,
		\begin{equation}\label{3.20}
		q^{as}p^{b}\equiv p^{b}\ne1\pmod {N_{n_{1}}}.
		\end{equation}
		On the contrary, we assume that $p^{b}\pmod {N_{n_{1}}}=1$, then there must exist some integer $l$, such that $p^{b}=lN_{n_{1}}+1$. Since $(p,N_{n_{1}})=d>1$, i.e. $d|p$ and $d|N_{n_{1}}$, we have $d|1$, which contradicts the fact that $d>1$. Therefore the claim follows.

		From \eqref{3.19} and \eqref{3.20}, we obtain that \eqref{3.18} holds, hence
		\begin{equation*}
		\lambda_{n_{1}}-\lambda_{n_{2}}\in\mathcal{Z}(\hat{\nu}_{1}).
		\end{equation*}
		Since $n_{1}$ and $n_{2}$ are arbitrary, we obtain that
		$\{\Lambda^{*}-\Lambda^{*}\}\setminus{0}\subset\mathcal{Z}(\hat{\nu}_{i})$. Obviously, from Lemma \ref{lem3.3}, we see that $E_{\Lambda^{*}}$ is an orthogonal set in $L^{2}(\mu)$ with $\#E_{\Lambda^{*}}=N$. By the arbitrariness of $N$, the proof is complete.
	\end{proof}

	\label{reference}

\end{document}